\documentclass[12pt, reqno]{amsart}
\usepackage{amsmath, amsthm, amscd, amsfonts, amssymb, mathtools, color, hyperref}
\usepackage[dvipsnames]{xcolor}
\usepackage{tikz}

\newtheorem{theorem}{Theorem}[section]
\newtheorem{lemma}[theorem]{Lemma}
\newtheorem{proposition}[theorem]{Proposition}

\theoremstyle{definition}

\newtheorem{example}[theorem]{Example}

\theoremstyle{remark}
\newtheorem{remark}[theorem]{Remark}
\numberwithin{equation}{section}

\begin{document}
	\setcounter{page}{1}
	
	\title[Singular difference graphs]
	{Singular difference graphs of vector spaces of square matrices}
	
	\author[Shrinath]{ Shrinath Hadimani }
	\address{Manipal Institute of Technology, Manipal Academy of Higher Education\\ 
		Manipal, India.}
	\email{shrinath.hadimani@manipal.edu}

	

	\begin{abstract}
		The singular difference graph, denoted by $\Gamma$, of the vector space of square matrices over a field 
		is a graph whose vertex set is the set of 
		all elements of the vector space, where two distinct vertices are adjacent if and only if the difference 
		of the corresponding matrices is singular. In this paper, we investigate fundamental 
		graph-theoretic properties of $\Gamma$, including connectivity, diameter, regularity,
		the Eulerian property, independence number, clique number, and domination
		number. We show that $\Gamma$ is a connected regular graph with diameter two. Over finite 
		fields, we obtain an explicit formula for the
		degree of each vertex and characterize precisely when $\Gamma$ is
		Eulerian. We determine the independence number and clique number and
		provide explicit constructions attaining these values using companion
		matrices of irreducible polynomials. We also construct an explicit
		dominating set, yielding an upper bound for the domination number.
	\end{abstract}

	\maketitle
	
	\section{Introduction}
	Graphs defined on algebraic structures, particularly vector spaces and
	matrix rings over finite fields, have been studied extensively to explore 
	interplay between algebraic properties and graph-theoretic parameters.
	Notable examples include zero-divisor graphs of rings
	\cite{Anderson}, commuting graphs of groups and rings
	\cite{Akbari, Iranmanesh}, and unit graphs of rings \cite{Ashrafi}. These
	graphs provide connections between algebraic properties and graph-theoretic
	parameters such as connectivity, diameter, clique number, and domination
	number.
	
	Graphs defined on matrix rings over fields have received particular
	attention. For instance, the commuting graph of square matrices, whose
	vertices are matrices and whose edges correspond to pairs of commuting
	matrices, has been studied with respect to its connectivity and diameter in 
	\cite{Akbari}. The non-commuting graph associated with a non-abelian group, in which two
	distinct noncentral elements are adjacent if and only if they do not commute,
	was introduced and studied in \cite{Abdollahi}. Zero-divisor graphs of finite rings have also 
	been studied, including those associated with matrix rings, with attention 
	to their structural and graph-theoretic properties \cite{Akbari-Mohammadian}. 
	More recently, compressed zero-divisor graphs associated with matrix
	rings over finite fields have been studied in terms of their structural
	and graph-theoretic properties \cite{Duric}. Graphs constructed from vector 
	spaces over finite fields have also been studied from various perspectives. 
	For example, the nonzero component graph define adjacency according to the
	presence of nonzero components in vectors \cite{Das-1, Das-2}. Adjacency
	preserving maps on matrices and operators have been investigated to
	understand transformations that preserve adjacency relations
	\cite{Petek-Semrl}. These studies demonstrate the usefulness of
	graph-theoretic methods in understanding algebraic structures associated
	with matrices.
	
	In this paper, we introduce the singular difference graph of the vector space 
	of square matrices over a field. In this graph, two distinct matrices
	are adjacent if and only if their difference is singular. We investigate
	graph-theoretic properties arising from this singularity relation.
	The study of independence and clique numbers is closely connected
	with the structure of linear subspaces of matrices and the use of
	companion matrices of irreducible polynomials. The domination number presents 
	a different type of problem, since its determination requires finding a smallest 
	set of matrices that meets the singularity condition with every matrix outside of 
	the set. We give an explicit construction of such a set, which provides an upper 
	bound for the domination number. The problem of determining the exact domination 
	number remains unresolved.
	
	\section{Definitions and preliminaries}\label{sec-Intro}
	In this section, we recall the basic definitions, notation, and preliminary results 
	from graph theory and linear algebra that will be used throughout this paper. We 
	begin with some standard concepts from graph theory.
	
	\subsection{Preliminaries on graph theory}
	
	A (simple) graph is an ordered pair $G=(V,E)$, where $V$ is a nonempty set of 
	vertices and $E \subseteq \{\{u,v\}:u,v\in V,\ u\neq v\}$ is the set of edges. 
	Two distinct vertices $u,v\in V$ are said to be adjacent if $\{u,v\}\in E$, in 
	which case we write $u\sim v$. If $\{u,v\}\notin E$, then $u$ and $v$ are said 
	to be nonadjacent, and we write $u\nsim v$. For a vertex $v \in V$, the \emph{degree}
	of $v$, denoted by $\deg(v)$, is equal to the number of vertices adjacent to it.
	A graph $G$ is said to be \emph{regular} if every vertex of $G$
	has the same degree. A graph is 
	said to be \emph{finite} if its vertex set is finite. Otherwise, it is called 
	\emph{infinite}. The \emph{complement} of a graph $G=(V,E)$ is the graph 
	$\overline{G}=(V,\overline{E})$ with the same vertex set as $G$, where two 
	distinct vertices are adjacent in $\overline{G}$ if and only if they are 
	nonadjacent in $G$. 
	
	A \emph{walk} in $G$ is a sequence of vertices $v_0,v_1,\ldots,v_k$ such that $v_{i-1}$ and $v_i$ are adjacent for every $i=1,2,\ldots,k$. Note that the vertices and edges may be repeated in a walk. A \emph{trail} is a walk in which no edge is repeated, although vertices may be repeated. A \emph{path} is a walk in which no vertex is repeated. The \emph{length of a path} is the number of edges in it. A graph $G$ is said to be \emph{connected} if, for every pair of
	vertices $u,v\in V$, there exists a path in $G$ joining $u$ and $v$. For two vertices $u,v\in V$, the \emph{distance} between $u$ and $v$,
	denoted by $d(u,v)$, is the length of a shortest path joining $u$ and
	$v$. The \emph{diameter} of a connected graph $G$, denoted by
	$\operatorname{diam}(G)$, is the maximum distance between any two
	vertices of $G$. A connected graph is called \emph{Eulerian} if it
	contains a closed trail that contains every edge of the graph exactly once.
	Equivalently, a connected graph is Eulerian if and only if every vertex
	has even degree.
	
	For a vertex $v \in V$, the \emph{open neighbourhood} of $v$ is defined by
	$N(v)=\{u\in V: u\sim v \}$. 
	A subset $\mathcal{I} \subseteq V$ is called an \emph{independent set} if no 
	two distinct vertices in $\mathcal{I}$ are adjacent. The maximum cardinality 
	of an independent set in $G$ is called the \emph{independence number} of $G$ 
	and is denoted by $\alpha(G)$. A subset $\mathcal{D} \subseteq V$ is called a 
	\emph{dominating set} of $G$ if every vertex in
	$V\setminus \mathcal{D}$ is adjacent to at least one vertex of $\mathcal{D}$. 
	The minimum cardinality of a dominating set is called the \emph{domination number} 
	of $G$ and is denoted by $\gamma(G)$. A \emph{clique} in a graph $G=(V,E)$ is a subset 
	$\mathcal{C} \subseteq V$ such that every pair of distinct vertices in $\mathcal{C} $ 
	is adjacent. The maximum cardinality of a clique in $G$ is called the 
	\emph{clique number} of $G$, denoted by
	$\omega(G)$. Throughout this paper, the cardinality of a set $X$ is denoted by $|X|$.

	\subsection{Preliminaries on finite fields}
	
	Let $\mathbb{F}$ be a field. We denote by $M_n(\mathbb{F})$ the vector space of all $n\times n$ matrices over $\mathbb{F}$. Furthermore, let $\mathbb{F}[x]$ denote the polynomial ring in the indeterminate $x$ with coefficients in $\mathbb{F}$. A nonconstant polynomial $p(x)\in\mathbb{F}[x]$
	is said to be \emph{irreducible over $\mathbb{F}$} if it cannot be expressed as 
	a product of two nonconstant polynomials in $\mathbb{F}[x]$. 
	
	It is well known that 
	for every prime power $q=p^k$, where $p$ is a prime and $k$ is a positive integer, 
	there exists a unique finite field, up to isomorphism, with exactly $q$ elements. 
	This field is denoted by $\mathbb{F}_q$. For every positive integer $n$, 
	there exists a monic irreducible polynomial of degree $n$ over $\mathbb{F}_q$. 
	Accordingly, let $p(x)=x^n+a_{n-1}x^{n-1}+\cdots+a_1x+a_0$
	be a monic irreducible polynomial over $\mathbb{F}_q$. The \emph{companion matrix} 
	of $p(x)$ is defined by
	$C=
	\begin{pmatrix}
		0 & 1 & 0 & \cdots & 0\\
		0 & 0 & 1 & \cdots & 0\\
		\vdots & \vdots & \vdots & \ddots & \vdots\\
		0 & 0 & 0 & \cdots & 1\\
		-a_0 & -a_1 & -a_2 & \cdots & -a_{n-1}
	\end{pmatrix}$.
	The characteristic polynomial of $C$ is $p(x)$. Since $p(x)$ is 
	irreducible, the minimal polynomial of $C$ also equals $p(x)$. Consequently, the matrices
	$I,\ C,\ C^2,\ \ldots,\ C^{n-1}$ are linearly independent over $\mathbb{F}_q$.
	
	\section{Singular difference graph of the vector space $M_n(\mathbb{F})$}\label{sec-Results}
	
	In this section, we define singular difference graph of the vector space $M_n(\mathbb{F})$, 
	where $n \geq 2$, and provide an example. 
	
	The \emph{singular difference graph} of the vector space $M_n(\mathbb{F})$, denoted by $\Gamma=(V,E)$, is the graph with 
	vertex set $V=M_n(\mathbb{F})$,
	and two distinct vertices $A,B \in M_n(\mathbb{F})$ are adjacent if and only if $\det(A-B)=0$.
	Equivalently, $A\sim B$ if and only if $\det(A-B)=0$.
	
	We now illustrate this definition with an example.
	
	\begin{example}\label{ex-1}
		Consider the singular difference graph of $M_2(\mathbb{F}_2)$, where $\mathbb{F}_2= \mathbb{Z}_2 = \{ 0, 1\}$ is the finite field with two elements, with addition and multiplication
		performed modulo $2$.
		Since $|M_2(\mathbb{F}_2)|=2^4=16$, the graph has 16 vertices (see Figure \ref{fig:sdg-2-2}). Notice that the zero 
		matrix is adjacent precisely to the nonzero singular
		matrices in $M_2(\mathbb{F}_2)$.

		\begin{center}
			\begin{figure*}[!t]
				\centering
				\resizebox{0.6\linewidth}{!}{%
					\begin{tikzpicture}[
						edge/.style={gray!55, line width=0.35pt}
						]
\coordinate (v0) at (0.0:5.2cm);
\coordinate (v1) at (22.5:5.2cm);
\coordinate (v2) at (45.0:5.2cm);
\coordinate (v3) at (67.5:5.2cm);
\coordinate (v4) at (90.0:5.2cm);
\coordinate (v5) at (112.5:5.2cm);
\coordinate (v6) at (135.0:5.2cm);
\coordinate (v7) at (157.5:5.2cm);
\coordinate (v8) at (180.0:5.2cm);
\coordinate (v9) at (202.5:5.2cm);
\coordinate (v10) at (225.0:5.2cm);
\coordinate (v11) at (247.5:5.2cm);
\coordinate (v12) at (270.0:5.2cm);
\coordinate (v13) at (292.5:5.2cm);
\coordinate (v14) at (315.0:5.2cm);
\coordinate (v15) at (337.5:5.2cm);

\draw[edge] (v0) -- (v1);
\draw[edge] (v0) -- (v2);
\draw[edge] (v0) -- (v3);
\draw[edge] (v0) -- (v4);
\draw[edge] (v0) -- (v5);
\draw[edge] (v0) -- (v8);
\draw[edge] (v0) -- (v10);
\draw[edge] (v0) -- (v12);
\draw[edge] (v0) -- (v15);
\draw[edge] (v1) -- (v2);
\draw[edge] (v1) -- (v3);
\draw[edge] (v1) -- (v4);
\draw[edge] (v1) -- (v5);
\draw[edge] (v1) -- (v9);
\draw[edge] (v1) -- (v11);
\draw[edge] (v1) -- (v13);
\draw[edge] (v1) -- (v14);
\draw[edge] (v2) -- (v3);
\draw[edge] (v2) -- (v6);
\draw[edge] (v2) -- (v7);
\draw[edge] (v2) -- (v8);
\draw[edge] (v2) -- (v10);
\draw[edge] (v2) -- (v13);
\draw[edge] (v2) -- (v14);
\draw[edge] (v3) -- (v6);
\draw[edge] (v3) -- (v7);
\draw[edge] (v3) -- (v9);
\draw[edge] (v3) -- (v11);
\draw[edge] (v3) -- (v12);
\draw[edge] (v3) -- (v15);
\draw[edge] (v4) -- (v5);
\draw[edge] (v4) -- (v6);
\draw[edge] (v4) -- (v7);
\draw[edge] (v4) -- (v8);
\draw[edge] (v4) -- (v11);
\draw[edge] (v4) -- (v12);
\draw[edge] (v4) -- (v14);
\draw[edge] (v5) -- (v6);
\draw[edge] (v5) -- (v7);
\draw[edge] (v5) -- (v9);
\draw[edge] (v5) -- (v10);
\draw[edge] (v5) -- (v13);
\draw[edge] (v5) -- (v15);
\draw[edge] (v6) -- (v7);
\draw[edge] (v6) -- (v9);
\draw[edge] (v6) -- (v10);
\draw[edge] (v6) -- (v12);
\draw[edge] (v6) -- (v14);
\draw[edge] (v7) -- (v8);
\draw[edge] (v7) -- (v11);
\draw[edge] (v7) -- (v13);
\draw[edge] (v7) -- (v15);
\draw[edge] (v8) -- (v9);
\draw[edge] (v8) -- (v10);
\draw[edge] (v8) -- (v11);
\draw[edge] (v8) -- (v12);
\draw[edge] (v8) -- (v13);
\draw[edge] (v9) -- (v10);
\draw[edge] (v9) -- (v11);
\draw[edge] (v9) -- (v12);
\draw[edge] (v9) -- (v13);
\draw[edge] (v10) -- (v11);
\draw[edge] (v10) -- (v14);
\draw[edge] (v10) -- (v15);
\draw[edge] (v11) -- (v14);
\draw[edge] (v11) -- (v15);
\draw[edge] (v12) -- (v13);
\draw[edge] (v12) -- (v14);
\draw[edge] (v12) -- (v15);
\draw[edge] (v13) -- (v14);
\draw[edge] (v13) -- (v15);
\draw[edge] (v14) -- (v15);

\fill (v0) circle (2pt);
\fill (v1) circle (2pt);
\fill (v2) circle (2pt);
\fill (v3) circle (2pt);
\fill (v4) circle (2pt);
\fill (v5) circle (2pt);
\fill (v6) circle (2pt);
\fill (v7) circle (2pt);
\fill (v8) circle (2pt);
\fill (v9) circle (2pt);
\fill (v10) circle (2pt);
\fill (v11) circle (2pt);
\fill (v12) circle (2pt);
\fill (v13) circle (2pt);
\fill (v14) circle (2pt);
\fill (v15) circle (2pt);

\node at (0.0:6.6cm) { $\begin{pmatrix}0&0\\0&0\end{pmatrix}$};
\node at (22.5:6.6cm) { $\begin{pmatrix}0&0\\0&1\end{pmatrix}$};
\node at (45.0:6.6cm) { $\begin{pmatrix}0&0\\1&0\end{pmatrix}$};
\node at (67.5:6.6cm) { $\begin{pmatrix}0&0\\1&1\end{pmatrix}$};
\node at (90.0:6.6cm) { $\begin{pmatrix}0&1\\0&0\end{pmatrix}$};
\node at (112.5:6.6cm) { $\begin{pmatrix}0&1\\0&1\end{pmatrix}$};
\node at (135.0:6.6cm) { $\begin{pmatrix}0&1\\1&0\end{pmatrix}$};
\node at (157.5:6.6cm) { $\begin{pmatrix}0&1\\1&1\end{pmatrix}$};
\node at (180.0:6.6cm) { $\begin{pmatrix}1&0\\0&0\end{pmatrix}$};
\node at (202.5:6.6cm) { $\begin{pmatrix}1&0\\0&1\end{pmatrix}$};
\node at (225.0:6.6cm) { $\begin{pmatrix}1&0\\1&0\end{pmatrix}$};
\node at (247.5:6.6cm) { $\begin{pmatrix}1&0\\1&1\end{pmatrix}$};
\node at (270.0:6.6cm) { $\begin{pmatrix}1&1\\0&0\end{pmatrix}$};
\node at (292.5:6.6cm) { $\begin{pmatrix}1&1\\0&1\end{pmatrix}$};
\node at (315.0:6.6cm) { $\begin{pmatrix}1&1\\1&0\end{pmatrix}$};
\node at (337.5:6.6cm) { $\begin{pmatrix}1&1\\1&1\end{pmatrix}$};
					\end{tikzpicture}%
				}
				\caption{The singular difference graph $\Gamma$ of $M_2(\mathbb{F}_2)$.
					Each vertex is labeled by the corresponding matrix.}
				\label{fig:sdg-2-2}
			\end{figure*}
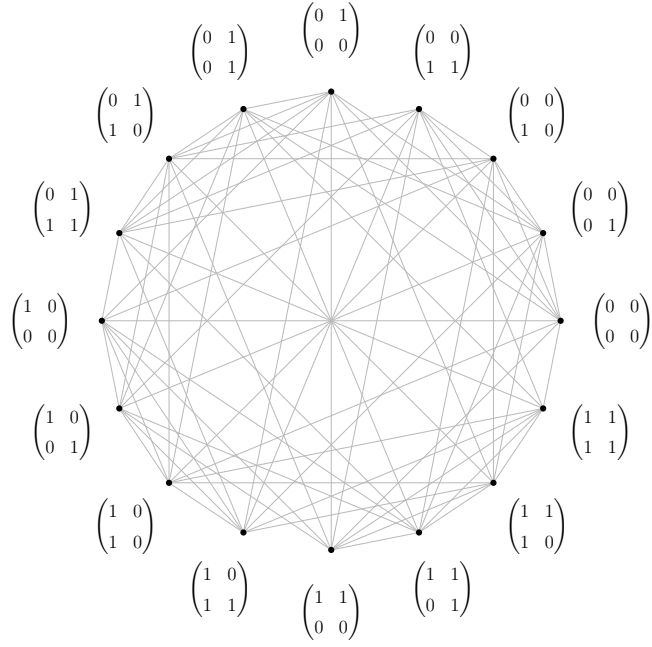    
		\end{center}	
	\end{example}

	The following proposition provides certain equivalent characterizations of adjacency in 
	a singular difference graph. This follows directly from the definition and standard 
	results from linear algebra. These characterizations will be used repeatedly throughout 
	this paper.
	
	\begin{proposition}\label{prop-1}
		Let $A,B\in M_n(\mathbb{F})$. Then, the following statements are equivalent:
		\begin{enumerate}
			\item $A\sim B$;
			\item $\det(A-B)=0$;
			\item $A-B$ is singular;
			\item $0$ is an eigenvalue of $A-B$;
			\item there exists a nonzero vector $x\in\mathbb{F}^n$ such that $Ax=Bx$.
		\end{enumerate}
	\end{proposition}
	
	\section{Properties of singular difference graphs}\label{Sec-Properties}
	
	In this section, we investigate graph-theoretic properties of the singular difference 
	graph $\Gamma$, including its connectedness, diameter, regularity, the Eulerian property, 
	independence number, clique number, and domination number. We begin with a lemma that 
	plays a fundamental role in establishing the connectivity of $\Gamma$.
	
	\begin{lemma}\label{lem-diameter}
		Let $A,B\in M_n(\mathbb{F})$, where $n \geq 2$. Then there exists a matrix
		$C \in M_n(\mathbb{F})$ such that $\det(A-C)=0=\det(B-C)$.
	\end{lemma}
	
	\begin{proof}
		Since $n \ge 2$, there exist two linearly independent vectors
		$x,y \in \mathbb{F}^n$. Choose vectors $v, u \in \mathbb{F}^n$ satisfying $v^Tx=1$, 
		$v^Ty=0$, $u^Tx=0$ and $u^Ty=1$.
		Define $C=Axv^T+Byu^T$. Then
		$Cx=Ax(v^Tx)+By(u^Tx)=Ax$
		and
		$Cy=Ax(v^Ty)+By(u^Ty)=By$.
		Therefore, by Proposition \ref{prop-1},
		$\det(A-C)=0=\det(B-C)$, which completes the proof.
	\end{proof}
	
	Using Lemma \ref{lem-diameter}, we establish the connectivity and diameter of $\Gamma$.
	
	\begin{theorem}\label{thm-connected}
		The singular difference graph $\Gamma$ of $ M_n(\mathbb{F})$ is connected and $\operatorname{diam}(\Gamma)=2$.
	\end{theorem}
	
	\begin{proof}
		Let $A,B\in M_n(\mathbb{F})$.
		If $A\sim B$, then the distance between $A$ and $B$ is $1$.
		Suppose that $A\nsim B$. By Lemma~\ref{lem-diameter}, there exists a matrix
		$C\in M_n(\mathbb{F})$ such that $A\sim C$ and $C\sim B$.
		Thus, there is a path of length $2$ joining $A$ and $B$. Hence every pair of
		vertices is connected by a path of length at most $2$, and therefore $\Gamma$ is
		connected and $\operatorname{diam}(\Gamma)\leq 2$. On the other hand, let $O$ and 
		$I$ denote the zero matrix and the identity matrix in $M_n(\mathbb{F})$, respectively. 
		Since $\det(O-I)=\det(-I)=(-1)^n\neq 0$, it follows that $O\nsim I$. Thus, $\Gamma$ 
		is not complete, and consequently $\operatorname{diam}(\Gamma)\neq 1$. Therefore, 
		$\operatorname{diam}(\Gamma)=2$.
		
	\end{proof}
	
	Next, we show that the singular difference graph $\Gamma$ is regular. Subsequently, we 
	determine its degree explicitly when the underlying field is finite.
	
	\begin{theorem}\label{thm-regular}
		The singular difference graph $\Gamma$ of $ M_n(\mathbb{F})$ is regular.
	\end{theorem}
	
	\begin{proof}
		Let $A\in M_n(\mathbb{F})$. By Proposition \ref{prop-1},
		$A\sim B$ if and only if $A-B$ is singular.
		Setting $S=B-A$, we obtain $A\sim B$ if and only if
		$B=A+S$, where $S$ is a singular matrix.
		Thus, $N(A)=\left\{A+S:S\in M_n(\mathbb{F}) \text{ is singular},\,S\neq O\right\}$,
		where $O$ denotes the zero matrix. Consequently,
		$\deg(A)
		=
		\left|\left\{S\in M_n(\mathbb{F}):S\text{ is singular}\right\}\right|-1$.
		Since this quantity is independent of the choice of $A$, every vertex has the same 
		degree. Hence, $\Gamma$ is regular.
	\end{proof}

	We now provide a spectral characterization of the neighbourhood of an invertible matrix. For $A \in M_n(\mathbb{F})$, $\sigma(A)$ denotes the set of all eigenvalues of $A$.
	
	\begin{theorem}\label{thm-invertible}
		Let $A\in M_n(\mathbb{F})$ be invertible. Then
		$N(A)=\{\,B\in M_n(\mathbb{F}) : 1\in\sigma(A^{-1}B)\,\}$.
	\end{theorem}
	
	\begin{proof}
		Let $B\in M_n(\mathbb{F})$. Since $A$ is invertible,
		\begin{align*}
			B\in N(A)
			&\iff A\sim B\\
			&\iff \exists\,x\in\mathbb{F}^n\setminus\{0\}\text{ such that }Ax=Bx\\
			&\iff \exists\,x\in\mathbb{F}^n\setminus\{0\}\text{ such that }A^{-1}Bx=x\\
			&\iff \exists\,x\in\mathbb{F}^n\setminus\{0\}\text{ such that }A^{-1}Bx=1\cdot x\\
			&\iff 1\in\sigma(A^{-1}B).
		\end{align*}
		Hence, $N(A)=\{\,B\in M_n(\mathbb{F}) : 1\in\sigma(A^{-1}B)\,\}$.
	\end{proof}
	
	\begin{remark}
		The neighbourhoods of the identity matrix $I$ and the zero matrix $O$
		admit simple spectral descriptions.
		
		\begin{enumerate}
			\item[(i)] By Theorem~\ref{thm-invertible}, the neighbourhood of the
			identity matrix is
			$N(I)=\{\,B\in M_n(\mathbb{F}):1\in\sigma(B)\,\}$.
			
			\item[(ii)] By Proposition~\ref{prop-1}(4), the neighbourhood of the
			zero matrix is
			$N(O)=\{\,B\in M_n(\mathbb{F}):0\in\sigma(B)\,\}$.
			Equivalently, $N(O)$ is precisely the set of all nonzero singular
			matrices in $M_n(\mathbb{F})$.
		\end{enumerate}
	\end{remark}
	
	We now restrict our attention to the finite field $\mathbb{F}_q$, where
	$q=p^k$ for some prime $p$ and positive integer $k$. In this setting, the degree of each
	vertex admits a closed-form expression.
	\medskip
	\begin{theorem}\label{Thm-degree of vertex}
		Let $\Gamma$ be the singular difference graph of
		$M_n(\mathbb{F}_q)$. Then, for every
		$A\in M_n(\mathbb{F}_q)$,
		$\deg(A)
		=
		q^{n^2}
		-
		\displaystyle	\prod_{j=0}^{n-1}(q^n-q^j)
		-1$.
	\end{theorem}
	
	\begin{proof}
		The total number of matrices in $M_n(\mathbb{F}_q)$ is $q^{n^2}$, while the number of
		invertible matrices is
		$\displaystyle \prod_{j=0}^{n-1}(q^n-q^j)$.
		Hence, the number of singular matrices equals
		$q^{n^2}
		-
		\displaystyle \prod_{j=0}^{n-1}(q^n-q^j)$.
		By the proof of Theorem \ref{thm-regular}, the degree of $A$ is one less than the 
		number of singular matrices in $M_n(\mathbb{F}_q)$. Therefore,
		$\deg(A)
		=
		q^{n^2}
		-
		\displaystyle \prod_{j=0}^{n-1}(q^n-q^j)-1$.
	\end{proof}
	
	We now characterize the Eulerian property of the singular difference graph
	$\Gamma$ of $M_n(\mathbb{F}_q)$ using Theorem \ref{Thm-degree of vertex}.
	
	\begin{theorem}
		Let $\Gamma$ be the singular difference graph of
		$M_n(\mathbb{F}_q)$. Then
		$\Gamma$ is Eulerian if and only if $q$ is odd.
	\end{theorem}
	
	\begin{proof}
		From Theorem \ref{thm-connected}, the graph $\Gamma$ is connected. Therefore, 
		by Euler's characterization, $\Gamma$ is Eulerian if and only if every
		vertex of $\Gamma$ has an even degree.
		Let $A\in M_n(\mathbb{F}_q)$. By
		Theorem \ref{Thm-degree of vertex},
		$\deg(A)
		=
		q^{n^2}
		-
		\displaystyle \prod_{j=0}^{n-1}(q^n-q^j)
		-1$. Suppose that $q$ is odd. Then $q^{n^2}$ is odd and
		$q^n-1$ is even, and hence
		$\displaystyle \prod_{j=0}^{n-1}(q^n-q^j)
		=
		(q^n-1)\displaystyle \prod_{j=1}^{n-1}(q^n-q^j)$
		is even. It follows that the degree of $A$ is the difference between an odd number and an even number, minus one, and hence is even. Thus every vertex of $\Gamma$ has even degree, and consequently
		$\Gamma$ is Eulerian.
		
		Conversely, suppose that $q$ is even. Then, $q^{n^2}$ is even. Since
		$q^n$ and $q^j$ are even for every $j=1,\ldots,n-1$, each factor
		$q^n-q^j$ is even. Hence
		$\displaystyle \prod_{j=0}^{n-1}(q^n-q^j)$
		is even. Therefore, the degree of $A$ is the difference between an even number and an even number, minus one, and hence is odd. 
		Thus every vertex of $\Gamma$ has odd degree; therefore  $\Gamma$ is 
		not Eulerian. Hence, $\Gamma$ is Eulerian if and only if $q$ is odd.
	\end{proof}
	
	\medskip
	Next, we determine the independence number of the singular
	difference graph $\Gamma$ of $M_n(\mathbb{F}_q)$. We first establish an upper bound for
	$\alpha(\Gamma)$ and then show that this bound is attained by explicitly
	constructing an independent set of $\Gamma$. We begin with the following lemma.
	
	\begin{lemma}\label{Lem-gcd}
		Let $p(x)\in\mathbb{F}[x]$ be a monic polynomial of degree $n$, and let
		$C$ be its companion matrix. If  $q(x)\in\mathbb{F}[x]$ satisfies $\gcd(p(x),q(x))=1$, 
		then $q(C)$ is invertible.
	\end{lemma}
	
	\begin{proof}
		Since $\gcd(p(x),q(x))=1$, Bézout's identity yields polynomials
		$a(x),b(x)\in\mathbb{F}[x]$ such that $a(x)p(x)+b(x)q(x)=1$.
		Evaluating this identity at the companion matrix $C$, we obtain
		$a(C)p(C)+b(C)q(C)=I$. Since $p(C)=0$ by the Cayley--Hamilton theorem, $b(C)q(C)=I$. 
		Hence, $q(C)$ is invertible, which completes the proof.
		
	\end{proof}
	
	\begin{theorem}\label{Thm-independence}
		Let $\Gamma$ be the singular difference graph of
		$M_n(\mathbb{F}_q)$. Then $\alpha(\Gamma)=q^n$.
	\end{theorem}
	
	\begin{proof}
		We first show that $\alpha(\Gamma)\le q^n$.
		There are exactly $q^n$ possible first rows of matrices in
		$M_n(\mathbb{F}_q)$. Hence, by the pigeonhole principle, among any
		$q^n+1$ matrices, there exist two matrices, say $A$ and $B$, having the
		same first row. Consequently, the first row of $A-B$ is the zero vector,
		and therefore $A-B$ is singular. Thus,
		$\det(A-B)=0$,
		which implies that $A\sim B$. Hence no independent set can contain more
		than $q^n$ vertices, and so
		$\alpha(\Gamma)\le q^n$.
		
		We prove the reverse inequality by constructing an independent set of cardinality $q^n$.
		Let
		$p(x)=x^n+a_{n-1}x^{n-1}+\cdots+a_1x+a_0$
		be an irreducible polynomial over $\mathbb{F}_q$, and let $C$ denote its
		companion matrix. Define
		$\mathcal{S}=
		\left\{
		\alpha_0I+\alpha_1C+\cdots+\alpha_{n-1}C^{\,n-1}
		:
		\alpha_i\in\mathbb{F}_q
		\right\}$.
		Since the minimal polynomial of $C$ is $p(x)$, the matrices
		$I,C,\ldots,C^{n-1}$
		are linearly independent over $\mathbb{F}_q$. Hence,
		$|\mathcal{S}|=q^n$.
		Let $A,B\in \mathcal{S}$ with $A\neq B$. Then
		$A-B=q(C)$,
		where $q(x)\in\mathbb{F}_q[x]$ is a nonzero polynomial with
		$\deg q<n$. Since $p(x)$ is irreducible and $\deg q<\deg p$, we have $\gcd(p(x),q(x))=1$. 
		Therefore, by Lemma~\ref{Lem-gcd}, $q(C)$ is invertible. It follows that
		$\det(A-B)\neq 0$,
		and consequently
		$A\nsim B$.
		Thus, $\mathcal{S}$ is an independent set of $\Gamma$, yielding
		$\alpha(\Gamma)\ge |\mathcal{S}|=q^n$.
		Combining the upper and lower bounds, we conclude that
		$\alpha(\Gamma)=q^n$.
	\end{proof}
	
	\begin{remark}
		Theorem \ref{Thm-independence} admits the following extremal interpretation. One can choose at most $q^n$ matrices from $M_n(\mathbb{F}_q)$ such that the difference of any two distinct matrices is invertible. Moreover, the bound $q^n$ is attained by the explicit construction given in the proof.
	\end{remark}
	
	We illustrate the construction in Theorem \ref{Thm-independence} with an example as follows.
	
	\begin{example}
		Let $n=2$ and $q=2$. Then, $\mathbb{F}_q=\mathbb{Z}_2$. Using the
		construction in Theorem~\ref{Thm-independence}, we determine a maximum
		independent set of the singular difference graph associated with
		$M_2(\mathbb{F}_2)$.
		The polynomial
		$p(x)=x^2+x+1$
		is irreducible over $\mathbb{F}_2$. Its companion matrix is
		$C=
		\begin{bmatrix}
			0 & 1 \\
			-1 & -1 
		\end{bmatrix}= \begin{bmatrix}
			0 & 1 \\
			1 & 1 
		\end{bmatrix}$.
		Hence,
		$\mathcal{S}=\left\{ \alpha_0I + \alpha_1C \colon \alpha_0, \alpha_1 \in \mathbb{F}_2
		\right\} =\left\{
		0,\,
		I,\,
		C,\,
		I+C
		\right\}$
		is an independent set of $\Gamma$, where $I+C = \begin{bmatrix}
			1 & 1 \\
			1 & 0 
		\end{bmatrix}$. Since $|\mathcal{S}|=2^2=4=\alpha(\Gamma)$,
		$\mathcal{S}$ is a maximum independent set. This is also evident from Figure \ref{fig:sdg-2-2} of Example \ref{ex-1}.
	\end{example}
	
	We now determine the clique number of the singular difference graph
	$\Gamma$.
	
	\begin{theorem}\label{clique number}
		Let $\Gamma$ be the singular difference graph of
		$M_n(\mathbb{F}_q)$. Then $\omega(\Gamma)=q^{n^2-n}$.
	\end{theorem}
	
	\begin{proof}
		Let $\mathcal{W}$ be the subspace of $M_n(\mathbb{F}_q)$ consisting of all matrices 
		whose first row is a zero vector. Any two distinct
		matrices in $\mathcal{W}$ have a nonzero difference whose first row is
		zero and hence are adjacent in $\Gamma$. Thus, $\mathcal{W}$ is a
		clique in $\Gamma$. Since the remaining $n-1$ rows of a matrix in $\mathcal{W}$ are 
		arbitrary, the number of matrices in $\mathcal{W}$ is $q^{n(n-1)} = q^{n^2 -n}$. 
		Therefore, $\omega(\Gamma) \geq q^{n^2-n}$.
		
		We now prove the reverse inequality. Let $\mathcal{S}$ be the subspace
		defined in the proof of Theorem~\ref{Thm-independence}, namely,
		$\mathcal{S}
		=
		\left\{
		\alpha_0I+\alpha_1C+\cdots+\alpha_{n-1}C^{n-1}
		:
		\alpha_i\in\mathbb{F}_q
		\right\}$,
		where $C$ is the companion matrix of the irreducible polynomial used
		there. By the construction in Theorem~\ref{Thm-independence}, every
		nonzero matrix in $\mathcal{S}$ is invertible.
		Let $e_1, e_2, e_3, \ldots e_{n}$ be the standard basis vectors of ${\mathbb{F}_q}^n$. 
		Then the first rows of $I,C,C^2,\ldots,C^{n-1}$ are
		$e_1^T,e_2^T,\ldots,e_n^T$, respectively. Hence, if
		$K=\alpha_0I+\alpha_1C+\cdots+\alpha_{n-1}C^{n-1}\in\mathcal{S}$,
		then the first row of $K$ is
		$(\alpha_0,\alpha_1,\ldots,\alpha_{n-1})$.
		Consequently, for every $A\in M_n(\mathbb{F}_q)$, there exists
		$K\in\mathcal{S}$ having the same first row as $A$. It follows that
		$A-K\in \mathcal{W}$, and hence
		$M_n(\mathbb{F}_q)=\mathcal{W}+\mathcal{S}$. 
		We now show that the sum is a direct sum. Let
		$A\in \mathcal{W}\cap\mathcal{S}$. Since $A \in \mathcal{W}$, its first row is zero. 
		On the
		other hand, since $A\in\mathcal{S}$, its first row is
		$(\alpha_0,\alpha_1,\ldots,\alpha_{n-1})$ for some
		$\alpha_0,\ldots,\alpha_{n-1}\in\mathbb{F}_q$. Hence
		$\alpha_0=\alpha_1=\cdots=\alpha_{n-1}=0$,
		so $A=O$. Therefore,
		$\mathcal{W}\cap\mathcal{S}=\{O\}$,
		and consequently
		$M_n(\mathbb{F}_q)=W\oplus\mathcal{S}$.
		It follows that the cosets of $\mathcal{S}$ form a partition of
		$M_n(\mathbb{F}_q)$, hence
		$M_n(\mathbb{F}_q)
		=
		\displaystyle	\bigcup_{X\in \mathcal{W}}(X+\mathcal{S})$.
		Since
		$|\mathcal{W}|=q^{n^2-n}$,
		there are exactly $q^{n^2-n}$ such cosets.
		
		We claim that each coset $X+\mathcal{S}$ is an independent set in
		$\Gamma$. Let $A,B\in X+\mathcal{S}$ be distinct. Then
		$A=X+K_1$, $B=X+K_2$
		for some distinct $K_1,K_2\in\mathcal{S}$. Thus,
		$A-B=K_1-K_2$.
		Since $\mathcal{S}$ is a subspace, $K_1-K_2$ is a nonzero element of
		$\mathcal{S}$. By the construction of $\mathcal{S}$, it is invertible.
		Therefore, $\det(A-B)\neq0$,
		and hence $A\nsim B$. Thus, each coset $X+\mathcal{S}$ is an
		independent set.
		We have therefore partitioned the vertex set of $\Gamma$ into
		$q^{n^2-n}$ independent sets. A clique can contain at most one vertex
		from each independent set. Hence,
		$\omega(\Gamma)\leq q^{n^2-n}$.
		Combining this with the lower bound, we obtain
		$\omega(\Gamma)=q^{n^2-n}$.
		
	\end{proof}
	
	We now give an explicit dominating set for the singular
	difference graph $\Gamma$ of $M_n(\mathbb{F}_q)$. Although the exact 
	value of the domination
	number $\gamma(\Gamma)$ is not determined here, the construction below
	provides an upper bound.
	We first recall a notation that is required in the proof. For a matrix
	$A=(a_{ij})\in M_n(\mathbb{F}_q)$, let $C_{ij}$ denote the cofactor
	corresponding to the entry $a_{ij}$, that is,
	$C_{ij}=(-1)^{i+j}\det(A_{ij})$,
	where $A_{ij}$ is the matrix obtained from $A$ by deleting its $i$-th row
	and $j$-th column. Recall that the determinant of $A$ can be expanded
	along its $i$-th row as
	$\displaystyle \det(A)=\sum_{j=1}^n a_{ij}C_{ij}$.
	In particular, when $i=n$, each cofactor $C_{nj}$ depends only on the
	first $n-1$ rows of $A$. We now construct an explicit dominating set.
	
	\begin{theorem}\label{thm-dominating}
		Let
		$\mathcal{D}=
		\{cE_{nj}:c\in\mathbb{F}_q,\ 1\leq j\leq n\}$,
		where $E_{nj}$ denotes the $n \times n$ matrix with $1$ in the $(n,j)$-entry
		and zeros elsewhere. Then $\mathcal{D}$ is a dominating set of $\Gamma$, and
		$|\mathcal{D}|=n(q-1)+1$, where $\Gamma$ is the singular
		difference graph of $M_n(\mathbb{F}_q)$.
	\end{theorem}
	
	\begin{proof}
		Let $A\in M_n(\mathbb{F}_q)\setminus \mathcal{D}$. We consider two cases according 
		to the rank of the first $n-1$ rows of $A$.
		Suppose that the first $n-1$ rows of $A$ are linearly dependent.
		Then $\operatorname{rank}(A)\leq n-1$, and hence $A$ is singular.
		Therefore, $A\sim O$. Since $O\in \mathcal{D}$, the vertex $A$ is adjacent to a vertex
		in $\mathcal{D}$.
		Now suppose that the first $n-1$ rows of $A$ are linearly independent.
		Then the cofactor vector
		$(C_{n1},\ldots,C_{nn})$
		is nonzero. Hence, there exists $j\in\{1,\ldots,n\}$ such that
		$C_{nj}\neq0$. Since only the $(n,j)$-entry changes when
		$cE_{nj}$ is subtracted from $A$, expansion of the determinant along
		the last row gives
		$\det(A-cE_{nj})=\det(A)-cC_{nj}$.
		Choose
		$c= \displaystyle \frac{\det(A)}{C_{nj}}\in\mathbb{F}_q$.
		Then
		$\det(A-cE_{nj})=0$, 
		so $A\sim cE_{nj}$.
		Thus, every vertex in $M_n(\mathbb{F}_q)\setminus \mathcal{D}$ is adjacent to
		a vertex in $\mathcal{D}$. Hence, $\mathcal{D}$ is a dominating set for $\Gamma$.
		Finally, since there are $n$ choices for $j$ and $q-1$ choices for
		$c\in\mathbb{F}_q \setminus \{0\}$, together with $O$, we have
		$|D|=n(q-1)+1$.
	\end{proof}
	
	\begin{remark}
		Theorem~\ref{thm-dominating} provides the upper bound $\gamma(\Gamma)\leq n(q-1)+1$.
		For the singular difference graphs of $M_2(\mathbb{F}_2)$ and
		$M_2(\mathbb{F}_3)$, the domination numbers can be verified by direct
		computation. Specifically,
		$\gamma(\Gamma)=3$ and
		$\gamma(\Gamma)=5$
		respectively. In both cases, the values attain the upper bound
		$n(q-1)+1$.
		However, the construction in
		Theorem~\ref{thm-dominating} does not establish that the bound is
		optimal. Determining the exact value of $\gamma(\Gamma)$ would require
		a matching lower bound showing that every dominating set has at least
		$n(q-1)+1$ vertices. Such a lower bound does not follow directly from
		the structural results obtained in this paper. Moreover, the
		standard graph-theoretic lower bounds, such as the minimum-degree
		bound, do not yield the required lower bound.
		Hence the exact value of $\gamma(\Gamma)$ in general, remains an
		interesting open problem.
	\end{remark}
	

	\bigskip
	\noindent
	{\bf Data availability:} No data were used for the research done in this article.\\
	\noindent
	{\bf Declarations:} The authors declare that there are no conflicts of interest in this work.

	\bibliographystyle{amsplain}

\end{document}